\documentclass[11pt,reqno]{amsart}

\usepackage{amsfonts, amssymb, amsthm, amsmath, bbm, wasysym, enumerate, amsxtra, latexsym, fullpage, amscd, graphics, epic, youngtab, sansmath, mathtools, multicol, url, hyperref, bbm, blindtext, wasysym, comment, datetime, stmaryrd}
\usepackage[shortlabels]{enumitem}
\usepackage[all,cmtip]{xy}
\usepackage[dvipsnames]{xcolor}
\usepackage[makeroom]{cancel}  
\usepackage{youngtab}
\usepackage{ytableau}
\usepackage{tikz}
\usetikzlibrary{cd}

\usepgflibrary{shapes.geometric}

\tikzstyle{V}=[draw, fill =black, circle, inner sep=0pt, minimum size=1.5pt]
\tikzstyle{C}=[draw, fill =white, circle, inner sep=0pt, minimum size=1.5pt]
\tikzstyle{over}=[draw=white,double=black,line width=2pt, double distance=.5pt]
\numberwithin{equation}{section}
\usepackage[margin = 0.9in]{geometry}
\theoremstyle{definition}

\newtheorem{theorem}{Theorem}[section]
\newtheorem{lemma}[theorem]{Lemma}
\newtheorem{proposition}[theorem]{Proposition}
\newtheorem{corollary}[theorem]{Corollary}

\newtheorem{remark}[theorem]{Remark}
\newtheorem{conjecture}[theorem]{Conjecture}
\newtheorem{ex}[theorem]{Example}

\def\<{\langle}
\def\>{\rangle}

\allowdisplaybreaks 
\tikzstyle directed=[postaction={decorate,decoration={markings,
    mark=at position .65 with {\arrow[arrowstyle]{stealth}}}}]
 \tikzstyle reverse directed=[postaction={decorate,decoration={markings,
    mark=at position .65 with {\arrowreversed[arrowstyle]{stealth};}}}]
    \usetikzlibrary{arrows,shapes,positioning}
\usetikzlibrary{decorations.markings}
\tikzstyle arrowstyle=[scale=1]

\title{Transitioning between tableaux and spider bases for Specht modules}
\author[M.S. Im and J. Zhu]{Mee Seong Im and Jieru Zhu}
\address{Department of Mathematical Sciences, United States Military Academy, West Point, NY 10996}
\curraddr{Department of Mathematics, United States Naval Academy, Annapolis, MD 21402, USA}
    \email{meeseongim@gmail.com (Im)}

\address{Department of Mathematics, SUNY at Buffalo,  Buffalo, NY 14260}
    \email{jieruzhu@buffalo.edu (Zhu)}

\begin{document}

\begin{abstract}
Regarding the Specht modules associated to the two-row partition $(n,n)$, we provide a combinatorial path model to study the transitioning matrix from the tableau basis to the $A_1$-web basis (i.e. cup diagrams), and prove that the entries in this matrix are positive in the upper-triangular portion with respect to  a certain partial order.
\end{abstract}
 
\maketitle 

%=============================================
\section{Introduction}

Webs first appear in the study of Schur--Weyl duality and $\mathfrak{sl}_2$-tensor invariants \cite{rumer1932valenztheorie}. Later on, these cup-like diagrams, or non-crossing matchings, form a basis of the Temperley-Lieb algebra \cite{TL71}, which plays a key role in studying polynomial representations for the quantum groups $U_q(\mathfrak{sl}_2)$. Afterwards it has generalizations to the $\mathfrak{sl}_3$-setting \cite{Kup96,KK99}. Webs have also gained recent interests in tensor categories \cite{EGNO15}, TQFT and knot invariants. A modern reconstruction of $\mathfrak{sl}_n$-webs done by \cite{CKM14} generalizes beyond $n=2,3$, therefore webs can be regarded as morphisms of the subcategory of $U_q(\mathfrak{sl}_n)$-modules generated by exterior powers of the natural module. Along this path, further webs have been invented for mixed tensors (both symmetric and exterior) \cite{TVW15} and for Lie superalgebras of type Q \cite{Br19,BDK20}.

Having these categorical connections in mind, this article focuses mainly on the combinatorial nature of the $\mathfrak{sl}_2$-webs, being viewed as basis elements of the Specht module $S^{\lambda}$ for $S_{2n}$, associated to the two row partition $\lambda=(n,n)$. To be precise, these are noncrossing matchings of $2n$ dots on the horizontal line (see Section~\ref{S2} for details). Conceptually in the quantum case, these are the Kazhdan-Lusztig (KL) basis of the parabolic Hecke module, where the action of the KL generators of the Hecke algebra is given by vertical stacking of cup diagrams. 

Another basis of $S^{\lambda}$ is of historical importance: it is the polytabloid basis in the classical construction of $S^{\lambda}$. Here the basis $\{v_T\}$ is indexed by the set $\mathcal{T}(n,n)$ of standard Young tableaux $T$ of shape $\lambda$ (see Section~\ref{S2} for details). The relationship between the webs and the polytabloids has been intensively investigated by Russell--Tymoczko \cite{RT19}, and is the starting point of this article. Not coincidentally, there is a set map $\phi$ between $\mathcal{T}(n,n)$ and the set of webs, which has a straightforward combinatorial definition (see Section~\ref{S2.2}). 

The connection between webs and $\mathcal{T}(n,n)$ are also indicated in geometric representation theory, where the top cohomology group of type A Springer fibers form a module for the symmetric group. Fix a finite-dimensional vector space $V$ and an element $x\in GL(V)$, these fibers consist of complete flags in $V$ subject to a further condition in linear algebra terms. We focus on the case when the Jordan type of $x$ is $(n,n)$, i.e., $x$ is conjugate to a block diagonal matrix with two equal-sized blocks. In this case, the top cohomology has a basis parametrized by the irreducible components of the Springer fibers, which are in bijection with all standard Young tableaux of shape $(n,n)$ (cf. \cite{Spa76,Ste88,Var79}). On the other hand, the cup diagrams record the geometric information of the corresponding irreducible components, via graphical calculus on their 
singular cohomology  by \cite{Fun03,SW12}. 
% intersection cohomology. 
Also see \cite{Im-Lai-Wilbert} on the usage of geometric and topological techniques to single out each irreducible component of two-row Springer fibers for all classical types. 

The basis setup of our question is as follows: if $W^{\lambda}$ is the $S_{2n}$-module spanned by the webs, and $S^{\lambda}$ is the $S_{2n}$-module spanned by polytabloids, then the module isomorphism $\rho:S^{\lambda}\to W^{\lambda}$ is unique up to a scalar. The previously introduced set map $\phi$ does not preserve the $S_{2n}$-action, therefore fails to serve as a candidate. Nevertheless, it establishes a reasonable identification between webs and $\mathcal{T}(n,n)$, for one to study properties of the map $\rho$. In this respect we use $w_T=\phi(v_T)$ to denote a web basic element.
 
Among all relevant combinatorics on $\mathcal{T}(n,n)$,  there is a partial order established by Russell--Tymoczko \cite{RT19} using the tableau graph (see Section~\ref{S2.3}). Russell--Tymoczko then transfer this combinatorics to the webs under the identification $\phi$. The question now is to study the entries $a_{ST}$ in the transitioning matrix $(a_{ST})$, where $\rho(v_T)=\sum_{s\in \mathcal{T}(n,n)}a_{ST}w_S$. A key result in \cite{RT19} is that after a proper scaling, $\rho$ could be chosen such that $(a_{ST})$ is unitriangular with respect to the partial order on $\mathcal{T}(n,n)$.

In \cite[Conjecture~5.8]{RT19}, Russell--Tymoczko further conjectured that the uppertriangular portion has positive entries. This was partially proved by \cite[Theorem~1.2]{Rho18}, so that the uppertriangular portion is known to be nonnegative.  Our main result is a proof of \cite[Conjecture~5.8]{RT19} (see Theorem~\ref{mainth}):

\begin{theorem}
The entries in the transitioning matrix satisfy $a_{ST}>0$ if and only if $S\leq T$.
\end{theorem}

We now give the result which could be the first step in analyzing the entries of the inverse matrix (cf. Proposition~\ref{explicitpreimage}).

\begin{proposition}
In the $S_{2n}$-module isomorphism (up to a scalar), the web $w$ is identified to such a  polytabloid $v_T$ so that both entries in each column of $T$ are connected by an arc in $w$, and the smaller entry is on top. Moreover, $v_T$ is well-defined by this condition. 
\end{proposition}

%=============================================
\section*{Acknowledgement}
This project started at the Summer Collaborators Program based at the School of Mathematics at the Institute for Advanced Study. We thank their hospitality in hosting our research group, their generous help throughout our stay at Princeton, and their financial support to facilitate this project.

We also thank the other two research members in our group, Chun-Ju Lai and Arik Wilbert, for their contribution to the project. Specifically, we thank A.W. for bringing to us the original problem and potential methods of attacking the problem; his broad knowledge on the subject matter has been our continuous go-to source for literature reference. We thank C.-J.L. for his sharp insight for pointing out several mistakes in our proofs and his suggestions for improvement, as well as coding resources for a portion of the diagrams in this article. This project would not have been successful without their engagement.

The authors also thank Jonathan Kujawa, Julianna Tymoczko, and Mikhail Khovanov for helpful conversations. M.S.I. also acknowledges Joseph Gamson, Eric Basque, Venkat R. Dasari, National Academy of Sciences, and Army Research Laboratory for supporting this project.

\section{Preliminaries}\label{S2}

\subsection{Polytabloid basis} The symmetric group $S_{2n}$ has conjugacy classes parametrized by partitions of $2n$. It has a well-known construction of irreducible modules, called Specht modules, associated to each partition. To be precise, a \emph{partition} $\lambda$ of $2n$ is a weakly decreasing sequence of nonnegative integers $\lambda=(\lambda_1,\lambda_2,\dots)$, such that $\lambda_i=0$ eventually, and $\sum_{i\in \mathbb{Z}_{> 0}}\lambda_i=2n$. Each partition $\lambda$ can be represented by a \emph{Young diagram}, which consists of $\lambda_i$ left justified boxes in row $i$, for each $i\in \mathbb{Z}_{> 0}$.

A presentation of $S_{2n}$ is given by generators which are simple transpositions $s_i$ $(1\leq i\leq 2n-1)$ subject to further well-known relations which we omit. A word $\sigma$ in $s_i$ is said to be \emph{reduced} if it could not be rewritten into another word with fewer letters, by using the defining relations in $S_n$. The \emph{length} $\ell(\sigma)$ of a permutation $\sigma$ is the number of letters in a reduced expression of $\sigma$. We refer the reader to \cite{Fu97} for more details, such as the fact that $\ell(\sigma)$ is independent of the choice of the reduced expression.

 One formulation of the Specht module $S^{\lambda}$ associated to the partition $\lambda$ is as follows. A \emph{Young tableau} of shape $\lambda$ is a filling of the Young diagram $\lambda$ with integers $1,2,\dots, 2n$. A \emph{standard} Young tableau is such that the entries increase along each column and each row. The symmetric group $S_{2n}$ acts on the set of all Young tableau by acting on its entries. We now work over the field of complex numbers, and fix the partition $\lambda=(n,n)$. The Specht module $S^{\lambda}$ is defined to be the $\mathbb{C}$-vector space spanned by $\{v_T\}$, where $T$ ranges over all standard Young tableaux. This basis is often referred to as the \emph{polytabloid basis}. For a given $T$ and a fixed $i$, the standard condition implies that one of the three cases is true: 1) $i$ and $i+1$ are adjacent entries in the same row, 2) $i$ and $i+1$ are adjacent entries in the same column, 3) $i$ is in a row underneath $i+1$. Define the action of $S_{2n}$ on $v_T$ as follows
  \begin{align*}
 s_i.v_T=\begin{cases}
 -v_T \hspace{.3 in} \text{if } i \text{ and } i+1 \text{ are in the same column of }T,\\
 v_{s_i.T}  \hspace{.3 in} \text{if } i \text{ is  in a row underneath } i+1. 
 \end{cases}
 \end{align*}
The action in the case of 1) is irrelevant for our main results, hence we omit it to spare the need of introducing the so-called Garnir relations. 

\begin{ex}\label{ex1}
\ytableausetup{smalltableaux}
\begin{align*}
T=\begin{ytableau} 1 & 2 & 4 & 7 \\ 3 & 5& 6 & 8 \end{ytableau}\:, \hspace{.3 in}
R=s_6.T=\begin{ytableau} 1 & 2 & 4 & 6 \\ 3 & 5& 7 & 8 \end{ytableau}\:, \hspace{.3 in}
s_7.v_T=-v_T,  \hspace{.3 in} s_6.v_T=v_R.
\end{align*}
\end{ex}

\subsection{Web basis}\label{S2.2}
Cup diagrams are special cases of webs for the Lie algebra $\mathfrak{sl}_2$ \cite{rumer1932valenztheorie, CKM14}. They have wide applications in invariant theory, tensor categories, knot theory and many related topics. In particular, a \emph{cup diagram} with $2n$ dots in a horizontal axis, is a noncrossing matching between these dots, where the arcs lie below the axis. We label the dots by integers $1,2,\dots, 2n$ from left to right. Recall $\lambda=(n,n)$. The Specht module has an alternative formulation, by letting $W^{\lambda}$ be the $\mathbb{C}$-vector space whose basis is indexed by all cup diagrams. The action of $S_{2n}$ is defined as follows: on a given web $w$, 
\begin{align}
s_i.w=\begin{cases}
-w \hspace{.1 in} &\text{if }i \text{ and }i+1 \text{ are joined by the same arc},\\
w+w' \hspace{.1 in} &\text{if }i \text{ is the right endpoint of an arc, and }i+1 \text{ is the left endpoint of another arc}.
\end{cases} \label{webaction}
\end{align}
Here $w'$ is the cup diagram with the same arcs in $w$ which are not incident to $i$ or $i+1$. For the two arcs $(a,i)$ and $(i+1,b)$ in $w$, where $a<i<i+1<b$, the remaining two arcs in $w'$ are defined to be $(a,b)$ and $(i,i+1)$.

\begin{ex}\label{ex2}
Consider 
\begin{align*}
&w=\begin{tikzpicture}[baseline={(0,-.3)}, scale = 0.8]
\draw[dotted] (-.25,0) -- (8.25,0) -- (8.25,-1.5) -- (-.25,-1.5) -- cycle;
\begin{footnotesize}
\node at (.5,.2) {$1$};
\node at (1.5,.2) {$2$};
\node at (2.5,.2) {$3$};
\node at (3.5,.2) {$4$};
\node at (4.5,.2) {$5$};
\node at (5.5,.2) {$6$};
\node at (6.5,.2) {$7$};
\node at (7.5,.2) {$8$};
\end{footnotesize}
\draw[thick] (.5,0) .. controls +(2.5,-1.5)  .. +(5,0);
\draw[thick] (1.5,0) .. controls +(.5,-.5)  .. +(1,0);
\draw[thick] (3.5,0) .. controls +(.5,-.5)  .. +(1,0);
\draw[thick] (6.5,0) .. controls +(.5,-.5)  .. +(1,0);
\end{tikzpicture}, \hspace{.3 in}
w'=\begin{tikzpicture}[baseline={(0,-.3)}, scale = 0.8]
\draw[dotted] (-.25,0) -- (8.25,0) -- (8.25,-1.5) -- (-.25,-1.5) -- cycle;
\begin{footnotesize}
\node at (.5,.2) {$1$};
\node at (1.5,.2) {$2$};
\node at (2.5,.2) {$3$};
\node at (3.5,.2) {$4$};
\node at (4.5,.2) {$5$};
\node at (5.5,.2) {$6$};
\node at (6.5,.2) {$7$};
\node at (7.5,.2) {$8$};
\end{footnotesize}
\draw[thick] (.5,0) .. controls +(3.5,-1.75)  .. +(7,0);
\draw[thick] (1.5,0) .. controls +(.5,-.5)  .. +(1,0);
\draw[thick] (3.5,0) .. controls +(.5,-.5)  .. +(1,0);
\draw[thick] (5.5,0) .. controls +(.5,-.5)  .. +(1,0);
\end{tikzpicture}.\\
&\text{Then } s_7.w=-w \mbox{ and }  s_6.w=w+w'. 
\end{align*}
\end{ex}

There is a well-known bijection between the set $\mathcal{T}(n,n)$ of standard Young tableaux of shape $(n,n)$ and cup diagrams with $2n$ dots. This map $\phi$ sends a given tableau $T$ to a web $w$, where all left endpoints in $w$ are entries in the first row of $T$. One can refer to \cite[items (n) and (ww)]{St99} for a detailed discussion of this material. In Examples~\ref{ex1} and \ref{ex2}, $\phi(T)=w$.

\subsection{The transitioning matrix}\label{S2.3}
We now introduce a unitrigularity result by Russell--Tymoczko \cite{RT19} which motivates our main results. By \cite{PPR09,RT11} $S^{\lambda}\simeq W^{\lambda}$. By Schur's lemma, this isomorphism is unique up to a scalar. This isomorphism $\rho$ is given explicitly as follows: let $T_0$ be the tableau with $1,2,\dots,2n$ down successive columns. Let $w_0=\phi(T_0)$, then $\rho(v_{T_0})=w_0$. Since $S^{\lambda}$ is irreducible, $\rho$ is defined on all standard polytabloids $v_T$. To avoid confusion we let $w_T=\phi(T)$. The transitioning matrix between the two bases $\{v_T\}$ and $\{w_T\}$ is defined as $(a_{ST})$, where $a_{ST}$ is the coefficient in $\rho(v_T)=\sum_{S\in \mathcal{T}(n,n)}a_{ST}w_S$.

In \cite{RT19}, Russell--Tymoczko introduced a partial order on the set $\mathcal{T}(n,n)$. This order is related to the Bruhat order on $S_{2n}$. First, the \emph{tableaux graph} $\Gamma^{\mathcal{T}(n,n)}$ has vertex set $\mathcal{T}(n,n)$, with a directed edge from $T$ to $R$, if $R=s_i.T$ for some $s_i$, and $i$ is in a row underneath $i+1$ in $T$. If this is the case, then the directed edge is labeled by $s_i$. For $T,S\in \mathcal{T}(n,n)$, $T\leq S$ if and only if there is a directed path from $T$ to $S$. We will be using the following property of the tableaux graph:

\begin{lemma}(\cite[Lemma~3.1, Lemma~3.8]{RT19})\label{rankedgraph}  Given $T\in \mathcal{T}(n,n)$, the number of edges in a path from $T_0$ to $T$ is independent of the path chosen, and is equal to $\ell(w)$, if $T=w.T_0$.
\end{lemma}

The main result of \cite{RT19} is as follows. 
\begin{theorem}(\cite[Theorem 5.5]{RT19})\label{unitriangularity}
The matrix $(a_{ST})$ is upper-triangular with ones along the diagonal.
\end{theorem}

They further conjectured:
\begin{conjecture}(\cite[Conjecture~5.8]{RT19})
\label{mainconjecture}
The entries $a_{ST}>0$ if and only if $S\leq T$.
\end{conjecture}

B. Rhoades proves the following result: 
\begin{theorem}(\cite[Theorem 1.2]{Rho18})
If $S\leq T$, the entry $a_{ST}$ is nonnegative.
\end{theorem}
It is worth noting that the technique of \cite{Rho18}, using Pl\"{u}cker relations among $2$-by-$2$ minors of a matrix, can be rephrased by a  diagrammatical language. We rephrase their main method as follows. For $\lambda=(n,n)$ let $M^{\lambda}$ be the vector space spanned by all matchings of $2n$ dots on a horizontal axis, with arcs below the axis. We note that these matchings may have crossings in them. Furthermore, the diagrams satisfy the following relations:
\begin{align}
\begin{tikzpicture}[baseline={(0,-.3)}, scale = 0.8]
\draw[dotted] (-.25,0) -- (4.25,0) -- (4.25,-.9) -- (-.25,-.9) -- cycle;
\begin{footnotesize}
\node at (.5,.2) {$a$};
\node at (1.5,.2) {$b$};
\node at (2.5,.2) {$c$};
\node at (3.5,.2) {$d$};
\end{footnotesize}
\draw[thick] (.5,0) .. controls +(1,-1)  .. +(2,0);
\draw[thick] (1.5,0) .. controls +(1,-1)  .. +(2,0);
\end{tikzpicture} =& \:
\begin{tikzpicture}[baseline={(0,-.3)}, scale = 0.8]
\draw[dashed] (-.25,0) -- (4.25,0) -- (4.25,-.9) -- (-.25,-.9) -- cycle;
\begin{footnotesize}
\node at (.5,.2) {$a$};
\node at (1.5,.2) {$b$};
\node at (2.5,.2) {$c$};
\node at (3.5,.2) {$d$};
\end{footnotesize}
\draw[thick] (.5,0) .. controls +(.5,-1)  .. +(1,0);
\draw[thick] (2.5,0) .. controls +(.5,-1)  .. +(1,0);
\end{tikzpicture} +  
\begin{tikzpicture}[baseline={(0,-.3)}, scale = 0.8]
\draw[dashed] (-.25,0) -- (4.25,0) -- (4.25,-.9) -- (-.25,-.9) -- cycle;
\begin{footnotesize}
\node at (.5,.2) {$a$};
\node at (1.5,.2) {$b$};
\node at (2.5,.2) {$c$};
\node at (3.5,.2) {$d$};
\end{footnotesize}
\draw[thick] (.5,0) .. controls +(1.5,-1)  .. +(3,0);
\draw[thick] (1.5,0) .. controls +(.5,-.5)  .. +(1,0);
\end{tikzpicture}\:.   \label{identityresolvecrossing}
\end{align}
We note a few things: 1) the dots $a,b,c,d$ need not to be adjacent; there could be dots and arcs between the ones shown in the picture. 2) One does not care if the new arcs create more crossings with the rest of the picture. In other words, these relations are viewed as global relations rather than local relations, in the sense that one only cares about what dots are connected to each other. The arcs are equivalent as long as the boundary points are fixed.

We will refer to the first summand in Eq (\ref{identityresolvecrossing}) as the ``VV'' term, and the second summand as the ``$\mathbb{V}$'' term.

\begin{remark} \label{localrelation} One can also obtain Eq (\ref{identityresolvecrossing}) by applying the well-known relation for $\mathfrak{sl}_2$ webs (see for example, \cite[Corollary~6.2.3]{CKM14}, for a modern treatment of this result):
\begin{align*}
\begin{tikzpicture}[baseline={(0,-.8)}, scale = 0.6]
\draw[thick] (.5,0) .. controls +(.65,-1)  .. +(1.3,-2);
\draw[thick] (0.5,-2) .. controls+(.65,1) .. +(1.3,2);
\end{tikzpicture} \hspace{.1 in} = \hspace{.1 in}
\begin{tikzpicture}[baseline={(0,-.8)}, scale = 0.6]
\draw[thick] (.5,0) .. controls +(.3,-1)  .. +(0,-2);
\draw[thick] (1.8,0) .. controls+(-.3,-1) .. +(0,-2);
\end{tikzpicture} \hspace{.1 in}+\hspace{.1 in}
\begin{tikzpicture}[baseline={(0,-.8)}, scale = 0.6]
\draw[thick] (.5,0) .. controls +(.5,-.7)  .. +(1,0);
\draw[thick] (.5,-2) .. controls +(.5,.7)  .. +(1,0);
\end{tikzpicture}\: . 
\end{align*}
However, a few details need to be addressed: 1) The $\mathfrak{sl}_2$ webs have upward orientations, which are omitted from the picture above when the context is clear; 2) This local relation does not interfere the rest of the picture. To see that Eq~(\ref{identityresolvecrossing}) is true regardless the movement of the arcs, as long as boundary points are fixed, one needs to further argue that it is invariant under all Reidemeister-like moves, modulo relations for $\mathfrak{sl}_2$-web. For these subtle reasons, we cite \cite{Rho18} for a self-contained combinatorial proof of Eq~(\ref{identityresolvecrossing}).
\end{remark}

The $S_{2n}$-module structure on $M^{\lambda}$ is defined as follows. For a matching $m$ and a simple transposition $s_i$, $1\leq i\leq 2n-1$,
\begin{align*}
s_i.m=\begin{cases}
-m, \hspace{.3 in} &\text{if }i \text{ and }i+1 \text{ are joined by the same arc},\\
\:\:\: m', \hspace{.3 in} &\text{otherwise}.
\end{cases}
\end{align*}
Here, $m'$ is the matching that interchanges the dots $i$ and $i+1$ in $m$. 
\begin{ex}\label{ex2} When $n=3$,
\begin{align*}
&m=\begin{tikzpicture}[baseline={(0,-.3)}, scale = 0.8]
\draw[dotted] (-.25,0) -- (6.25,0) -- (6.25,-1) -- (-.25,-1) -- cycle;
\begin{footnotesize}
\node at (.5,.2) {$1$};
\node at (1.5,.2) {$2$};
\node at (2.5,.2) {$3$};
\node at (3.5,.2) {$4$};
\node at (4.5,.2) {$5$};
\node at (5.5,.2) {$6$};
\end{footnotesize}
\draw[thick] (.5,0) .. controls +(1,-1)  .. +(2,0);
\draw[thick] (1.5,0) .. controls +(1,-1)  .. +(2,0);
\draw[thick] (4.5,0) .. controls +(.5,-1)  .. +(1,0);
\end{tikzpicture}, \\ 
&m_1=\begin{tikzpicture}[baseline={(0,-.3)}, scale = 0.8]
\draw[dotted] (-.25,0) -- (6.25,0) -- (6.25,-1) -- (-.25,-1) -- cycle;
\begin{footnotesize}
\node at (.5,.2) {$1$};
\node at (1.5,.2) {$2$};
\node at (2.5,.2) {$3$};
\node at (3.5,.2) {$4$};
\node at (4.5,.2) {$5$};
\node at (5.5,.2) {$6$};
\end{footnotesize}
\draw[thick] (.5,0) .. controls +(1,-1)  .. +(2,0);
\draw[thick] (1.5,0) .. controls +(1.5,-1)  .. +(3,0);
\draw[thick] (3.5,0) .. controls +(1,-1)  .. +(2,0);
\end{tikzpicture}, \hspace{.2 in}
\mbox{ and } 
\hspace{.2 in}
m_2=\begin{tikzpicture}[baseline={(0,-.3)}, scale = 0.8]
\draw[dotted] (-.25,0) -- (6.25,0) -- (6.25,-1) -- (-.25,-1) -- cycle;
\begin{footnotesize}
\node at (.5,.2) {$1$};
\node at (1.5,.2) {$2$};
\node at (2.5,.2) {$3$};
\node at (3.5,.2) {$4$};
\node at (4.5,.2) {$5$};
\node at (5.5,.2) {$6$};
\end{footnotesize}
\draw[thick] (.5,0) .. controls +(1.5,-1)  .. +(3,0);
\draw[thick] (1.5,0) .. controls +(.5,-.5)  .. +(1,0);
\draw[thick] (4.5,0) .. controls +(.5,-1)  .. +(1,0);
\end{tikzpicture}, \\
& s_5.m=-m, \hspace{.5 in} s_4.m=m_1, \hspace{.5 in} s_1.m=s_3.m=m_2. 
\end{align*}
\end{ex}

A key argument in \cite[Eq~(2.9)]{Rho18} can be summarized by the following result.
\begin{lemma}(\cite[Eq~(2.9)]{Rho18})
There is an $S_{2n}$-module isomorphism $W^{\lambda}\simeq M^{\lambda}$, sending a (noncrossing) cup diagram to the same diagram in $M^{\lambda}$.
\end{lemma}

\section{The upper-triangular entries are positive}
The goal of this section is to prove Conjecture~\ref{mainconjecture} (\cite[Conjecture~5.8]{RT19}).

\subsection{A combinatorial formula}
We first establish a combinatorial formula for computing the coefficients $a_{ST}$ in Theorem~\ref{unitriangularity}. For $\lambda=(n,n)$, note that $S^{\lambda}\simeq W^{\lambda}\simeq M^{\lambda}$, and the isomorphism is unique up to a scalar. The first step is to express the image of a polytabloid $v_T\in S^{\lambda}$ as a simple matching in $M^{\lambda}$, which may contain crossings. An example will be given at the end of this section.

\begin{lemma}\label{firstisomorphism}
The isomorphism $f: S^{\lambda}\to M^{\lambda}$ sends $v_T$ to the matching $m$, where the entries in the same column of $T$ are joined by the same arc in $m$.
\end{lemma}
\begin{proof}
We induct on the partial order on $\mathcal{T}(n,n)$. The bases case when $T=T_0$ follows from the fact that $\rho(v_{T_0})=w_0$ mentioned in the beginning of Section~2.3. Suppose there is a path from $T_0$ to $S$ in the tableau graph $\Gamma^{\mathcal{T}(n,n)}$, where $T$ is the tableau immediately before $S$ in this sequence. Then $S$ is of lower order than $T$, and $S=s_i.T$ for some $i$, $1\leq i\leq 2n-1$. By the induction hypothesis, the statement holds for $T$. By the definition of $\Gamma^{\mathcal{T}(n,n)}$, $i$ is in a row beneath $i+1$ in $T$. Hence the arcs connected to $i$ and $i+1$ looks like the following in $f(v_T)\in M^{\lambda}$:
\begin{align}
m_1=\begin{tikzpicture}[baseline={(0,-.3)}, scale = 0.8]
\draw[dotted] (-.25,0) -- (4.25,0) -- (4.25,-.9) -- (-.25,-.9) -- cycle;
\begin{footnotesize}
\node at (1.5,.2) {$i$};
\node at (2.5,.2) {$i+1$};
\node at (1,-.5) {$a$};
\node at (3,-.5) {$b$};
\end{footnotesize}
\draw[thick] (0,-.5) .. controls +(0.5,0)  .. +(1.5,.5);
\draw[thick] (4,-.5) .. controls +(-0.5,0)  .. +(-1.5,.5);
\end{tikzpicture}\:, 
\end{align}
Therefore $f(v_S)=f(s_i.v_T)=s_i.f(v_T)$ looks like the following, taking into account the action of $s_i$ on $M^{\lambda}$ by permuting the strands:
\begin{align} 
m_2=\begin{tikzpicture}[baseline={(0,-.3)}, scale = 0.8]
\draw[dotted] (-.25,0) -- (4.25,0) -- (4.25,-.9) -- (-.25,-.9) -- cycle;
\begin{footnotesize}
\node at (1.5,.2) {$i$};
\node at (2.5,.2) {$i+1$};
\node at (1,-.3) {$a$};
\node at (3,-.3) {$b$};
\end{footnotesize}
\draw[thick] (0,-.5) .. controls +(1.5,0)  .. +(2.5,.5);
\draw[thick] (4,-.5) .. controls +(-1.5,0)  .. +(-2.5,.5);
\end{tikzpicture}\:. 
\end{align}
But this is exactly the matching prescribed for $v_S$ under the statement of the lemma: suppose $x$ and $i$ are in the same column of $T$; also, $i+1$ and $y$ are in another column of $T$. Then $x$ and $i+1$ are in the same column of $S$, connected by arc $a$ in $m_2$. Also, $i$ and $y$ are in the same column of $S$, connected by arc $b$ in $m_2$.
\end{proof}

\begin{remark}\label{fgeneral}
Note that the map $f$ can also be defined on all tableaux, not necessarily the standard one. A priori it may no longer preserve the $S_{2n}$-module structure, but one can view it simply as a set map from the set of Young tableaux to the set of matchings. In a later argument where only combinatorics is concerned, we will loosen our restriction and use this broader definition of $f$.
\end{remark}

Starting with a matching $m\in M^{\lambda}$, we now define a \emph{crossing-resolving} graph based on $m$. To draw such a graph, one starts with a single vertex $m$. Choose a crossing in $m$, and apply Eq (\ref{identityresolvecrossing}) to obtain two terms in the sum. Out of $m$ one further draws two arrow, by choosing a one labeled by ``VV'', whose target is the VV term in Eq (\ref{identityresolvecrossing}), as a sumand in $m$. The other arrow is labeled by ``$\mathbb{V}$'' and its target is the corresponding $\mathbb{V}$ term in Eq (\ref{identityresolvecrossing}). 

The resulting graph now has three vertices: one source (with outgoing arrows only) and two sinks (with incoming arrows only). For each sink, one further draws two outgoing arrows, labeled by VV and $\mathbb{V}$ each, based on the procedures given above. The graph continues as each sink is followed by two more arrows. The graph is complete when each sink is crossing-less. 

Note that this procedure eventually stops: based on Remark~\ref{localrelation}, the target of each newly created arrow must have fewer crossings compared to the source of the same arrow, hence the number of crossings eventually becomes zero. 

Also note that given a matching $m$, such a crossing-resolving graph is not unique. At each stage, one has a choice of which crossing to resolve. 

\begin{lemma}\label{formula}
Given a matching $m\in M^{\lambda}$ and a crossing-resolving graph $\Gamma$ for $m$. Let $\mathcal{S}$ be the set of matchings which appear as sinks in $\Gamma$. Then $m= \sum_{\mu \in S} a_{\mu} \mu$, where $a_{\mu}$ is the number of times $\mu$ occurs as a sink in $\Gamma$.
\end{lemma}
\begin{proof}
This is because each vertex is a sum of the two terms which are targets of its outgoing arrows.
\end{proof}

\begin{ex}
When $n=3$,
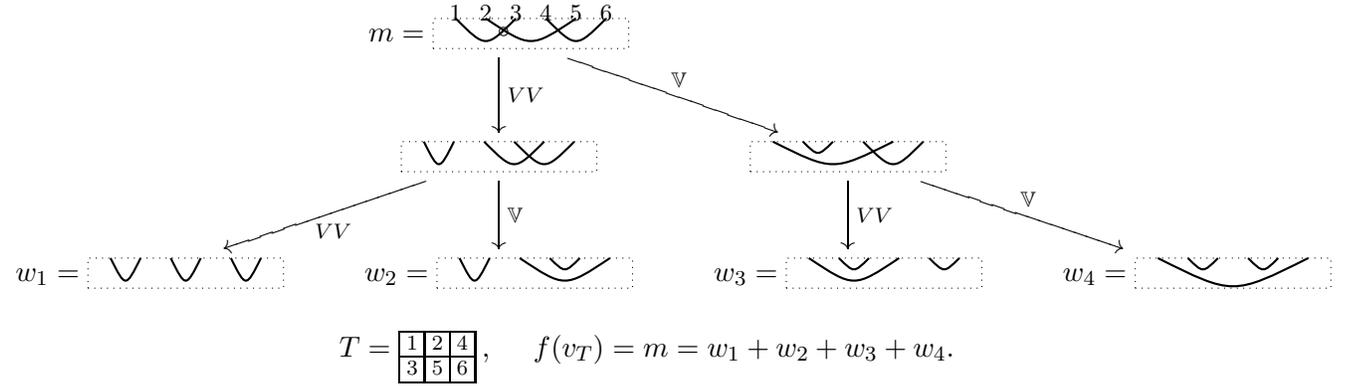
\begin{center}
\xymatrix{
& m=\begin{tikzpicture}[baseline={(0,-.3)}, scale = 0.4]
\draw[dotted] (-.25,0) -- (6.25,0) -- (6.25,-1) -- (-.25,-1) -- cycle;
\begin{footnotesize}
\node at (.5,.2) {$1$};
\node at (1.5,.2) {$2$};
\node at (2.5,.2) {$3$};
\node at (3.5,.2) {$4$};
\node at (4.5,.2) {$5$};
\node at (5.5,.2) {$6$};
\node at (2.1,-.45) {$\circ$};
\end{footnotesize}
\draw[thick] (.5,0) .. controls +(1,-1)  .. +(2,0);
\draw[thick] (1.5,0) .. controls +(1.5,-1)  .. +(3,0);
\draw[thick] (3.5,0) .. controls +(1,-1)  .. +(2,0);
\end{tikzpicture} \ar[d]^{VV} \ar[dr]^{\mathbb{V}} &&\\
& 
\begin{tikzpicture}[baseline={(0,-.3)}, scale = 0.4]
\draw[dotted] (-.25,0) -- (6.25,0) -- (6.25,-1) -- (-.25,-1) -- cycle;
\draw[thick] (.5,0) .. controls +(.5,-1)  .. +(1,0);
\draw[thick] (2.5,0) .. controls +(1,-1)  .. +(2,0);
\draw[thick] (3.5,0) .. controls +(1,-1)  .. +(2,0);
\end{tikzpicture}  \ar[dl]^{VV} \ar[d]^{\mathbb{V}}
 &
\begin{tikzpicture}[baseline={(0,-.3)}, scale = 0.4]
\draw[dotted] (-.25,0) -- (6.25,0) -- (6.25,-1) -- (-.25,-1) -- cycle;
\draw[thick] (.5,0) .. controls +(2,-1)  .. +(4,0);
\draw[thick] (1.5,0) .. controls +(.5,-.5)  .. +(1,0);
\draw[thick] (3.5,0) .. controls +(1,-1)  .. +(2,0);
\end{tikzpicture} \ar[d]^{VV}  \ar[dr]^{\mathbb{V}} &
  \\
w_1=\begin{tikzpicture}[baseline={(0,-.3)}, scale = 0.4]
\draw[dotted] (-.25,0) -- (6.25,0) -- (6.25,-1) -- (-.25,-1) -- cycle;
\draw[thick] (.5,0) .. controls +(.5,-1)  .. +(1,0);
\draw[thick] (2.5,0) .. controls +(.5,-1)  .. +(1,0);
\draw[thick] (4.5,0) .. controls +(.5,-1)  .. +(1,0);
\end{tikzpicture}   
   &
w_2=\begin{tikzpicture}[baseline={(0,-.3)}, scale = 0.4]
\draw[dotted] (-.25,0) -- (6.25,0) -- (6.25,-1) -- (-.25,-1) -- cycle;
\draw[thick] (.5,0) .. controls +(.5,-1)  .. +(1,0);
\draw[thick] (2.5,0) .. controls +(1.5,-1)  .. +(3,0);
\draw[thick] (3.5,0) .. controls +(.5,-.5)  .. +(1,0);
\end{tikzpicture}    
   &  
   w_3=\begin{tikzpicture}[baseline={(0,-.3)}, scale = 0.4]
\draw[dotted] (-.25,0) -- (6.25,0) -- (6.25,-1) -- (-.25,-1) -- cycle;
\draw[thick] (.5,0) .. controls +(1.5,-1)  .. +(3,0);
\draw[thick] (1.5,0) .. controls +(.5,-.5)  .. +(1,0);
\draw[thick] (4.5,0) .. controls +(.5,-.5)  .. +(1,0);
\end{tikzpicture} & 
w_4=\begin{tikzpicture}[baseline={(0,-.3)}, scale = 0.4]
\draw[dotted] (-.25,0) -- (6.25,0) -- (6.25,-1) -- (-.25,-1) -- cycle;
\draw[thick] (.5,0) .. controls +(2.5,-1.25)  .. +(5,0);
\draw[thick] (1.5,0) .. controls +(.5,-.5)  .. +(1,0);
\draw[thick] (3.5,0) .. controls +(.5,-.5)  .. +(1,0);
\end{tikzpicture}   
}
\begin{align*}
T=\begin{ytableau} 1 & 2 & 4  \\ 3 & 5& 6  \end{ytableau}\:, \hspace{.2 in} f(v_T)=m=w_1+w_2+w_3+w_4.
\end{align*}
\end{center}
In this example, the sum is multiplicity-free. We did not list an example with repeated terms due to the limitation of space.
\end{ex}

\subsection{Proof of Conjecture~\ref{mainconjecture}}
The main idea of proving Conjecture~\ref{mainconjecture} is to establish one crossing-resolving graph in which the desired web occurs as a sink. We need one lemma regarding the combinatorics of standard Young tableaux. 

\begin{lemma}\label{entries}
Let $S\to T$ be a directed edge in the tableau graph. Let $a_i$ be the $i$-th entry in the upper row of $S$, and $b_i$ the $i$-th entry in the upper row of $T$, then $a_i\geq b_i$ for all $1\leq i\leq n$.
\end{lemma} 
\begin{proof}
This follows from the fact that $T=s_i.S$ where $s_i$ takes an integer $i$ from the second row of $S$ and swaps it with $i+1$ from the first row of $S$. Therefore each entry in the first row gets replaced by a smaller entry or remains unchanged from $S$ to $T$. 
\end{proof}

\begin{corollary}\label{entries2}
If $S\leq T$ in the tableau graph, and $a_i$ and $b_i$ are defined as in Lemma~\ref{entries}, then $a_i\geq b_i$ for all $1\leq i\leq n$.
\end{corollary}

\begin{conjecture} The opposite direction in Corollary~\ref{entries2} is also true:  if  $a_i$ and $b_i$ are the entries in the first row of $S$ and $T$, respectively, and $a_i\geq b_i$ for all $1\leq i\leq n$, then $S\leq T$.
\end{conjecture}

Let $S,T$ be two standard Young tableaux of shape $\lambda=(n,n)$, whose entries are $a_i$ and $b_i$ in the first row respectively $(1\leq i\leq n)$, from left to right. We have the following.
\begin{proposition} \label{existspath}
Assume $a_i\geq b_i$. Let $m=f(v_T)\in M^{\lambda}$ be the matching given in Lemma~\ref{firstisomorphism}, and $w=\phi(S)$ be the web defined in Section~\ref{S2.2}. Then there exists one crossing-resolving graph for $m$ in which $w$ occurs as a sink.
\end{proposition}

\begin{proof}
We prove by induction on $n$.  Let $a$ be the upper right entry of $T$. Since $a$ is rightmost among all left endpoints, the vertices  $a+1,\dots,2n-1$ must all be right endpoints of other arcs, and the arcs incident to $a+1,\dots,2n-1$ must intersect the arc $(a,2n)$ exactly once. If these arcs have further intersections among themselves, one can move the intersections outside the arc $(a,2n)$. On the other hand, let $b$ be the upper right entry of $S$. Then $(b,b+1)$ must be a cup in $S$, because there are no left endpoint to the right of $b$. These observations allow us to draw both $T$, $S$, and their corresponding diagrams $m$ and $w$, as follows.  Here we only display the rightmost region of the diagrams, and use $b'=b+1$, $b''=b+2$, $\underline{b}=b-1$. The notation is similar for $a$. The dotted boxes are filled with consecutive integers:

\begin{align} 
&T= \begin{ytableau}
*(red) & *(red)  &*(red) & *(red) & *(red) & *(red) & *(red)& *(red) & a\\
*(red) & *(red)  &*(red) & *(red) & a' & \cdot  & \cdot & \cdot & 2n\\
\end{ytableau}\:, \hspace{.2 in} 
S= \begin{ytableau}
*(blue) & *(blue)  &*(blue) & *(blue) & *(blue) & *(blue) & *(blue) & *(blue) & b\\
*(blue) & *(blue) &*(blue) & *(blue) & b' & \cdot  & \cdot & \cdot & 2n\\
\end{ytableau}\: ,
\\
&m= \begin{tikzpicture}[baseline={(0,-.3)}, scale = 0.8]
\draw[dotted] (-.25,0) -- (5,0) -- (5,-1.7) -- (-.25,-1.7) -- cycle;
\begin{footnotesize}
\node at (.8,.2) {$a$};
\node at (1.5,-.2) {\dots};
\node at (1.9,.2) {$b$};
\node at (2.3,.2) {$b'$};
\node at (2.7,-.2) {\dots};
\node at (4.1,.2) {$2n$};
\node at (1.35,-.55) {$x$};
\node at (1.65,-.45) {$\circ$};
\end{footnotesize}
\draw[thick,color=red] (.8,0) .. controls +(2,-1) .. +(3.2,0);
\draw[thick] (.5,-.6) .. controls +(.4,0) .. +(.6,.6);
\draw[thick] (.5,-.9) .. controls +(.8,0) .. +(1.4,.9);
\draw[thick] (.5,-1.2) .. controls +(.9,0) .. +(1.7,1.2);
\draw[thick] (.5,-1.5) .. controls +(1.5,0) .. +(3,1.5);
\end{tikzpicture}\:, \hspace{.3 in}
w= \begin{tikzpicture}[baseline={(0,-.3)}, scale = 0.8]
\draw[dotted] (-.25,0) -- (5,0) -- (5,-1.7) -- (-.25,-1.7) -- cycle;
\begin{footnotesize}
\node at (.9,.2) {$a$};
\node at (1.1,-.2) {\dots};
\node at (1.9,.2) {$b$};
\node at (2.3,.2) {$b'$};
\node at (3,-.2) {\dots};
\node at (4.1,.2) {$2n$};
\end{footnotesize}
\draw[thick,color=red] (1.9,0) .. controls +(.15,-.5) .. +(.3,0);
\draw[thick] (.5,-1.2) .. controls +(1.1,0) .. +(2.1,1.2);
\draw[thick] (.5,-1.5) .. controls +(1.7,0) .. +(3.4,1.5);
\end{tikzpicture}  \: .  \label{mw}
\end{align}
% \node at (2,-.85) {$y$};
% \node at (1.9,-.55) {$\circ$};
We resolve the crossings $m$ intersecting the arc $(a,2n)$, in the following fashion: first resolve the series of crossings to the left of $x$, in the counterclockwise order, each using the "VV" move. The result is the following intermediate diagram $m_1$ below. Next, resolve the crossings to the right of $y$, in the clockwise order, each using the "$\mathbb{V}$" move. The result is the diagram $m'$ below. 
\begin{align*}
m_1= \begin{tikzpicture}[baseline={(0,-.3)}, scale = 0.8]
\draw[dotted] (-.25,0) -- (5,0) -- (5,-1.7) -- (-.25,-1.7) -- cycle;
\begin{footnotesize}
\node at (.8,.2) {$a$};
\node at (1.1,-.2) {\dots};
\node at (1.9,.2) {$b$};
\node at (2.3,.2) {$b'$};
\node at (2.7,-.2) {\dots};
\node at (4.1,.2) {$2n$};
\node at (2.1,-.55) {$y$};
\node at (2.1,-.2) {$\circ$};
\end{footnotesize}
\draw[thick,color=red] (1.9,0) .. controls +(1.1,-1) .. +(2.2,0);
\draw[thick] (.5,-.6) .. controls +(.15,0) .. +(.3,.6);
\draw[thick] (.5,-.9) .. controls +(.5,0) .. +(1.1,.9);
\draw[thick] (.5,-1.2) .. controls +(.9,0) .. +(1.7,1.2);
\draw[thick] (.5,-1.5) .. controls +(1.5,0) .. +(3,1.5);
\end{tikzpicture}\:\, \hspace{.2 in} m'= \begin{tikzpicture}[baseline={(0,-.3)}, scale = 0.8]
\draw[dotted] (-.25,0) -- (5,0) -- (5,-1.7) -- (-.25,-1.7) -- cycle;
\begin{footnotesize}
\node at (.9,.2) {$a$};
\node at (1.1,-.2) {\dots};
\node at (1.9,.2) {$b$};
\node at (2.3,.2) {$b'$};
\node at (3,-.2) {\dots};
\node at (4.1,.2) {$2n$};
\end{footnotesize}
\draw[thick,color=red] (1.9,0) .. controls +(.15,-.5) .. +(.3,0);
\draw[thick] (.5,-.6) .. controls +(.15,0) .. +(.3,.6);
\draw[thick] (.5,-.9) .. controls +(.5,0) .. +(1.1,.9);
\draw[thick] (.5,-1.2) .. controls +(1.1,0) .. +(2.1,1.2);
\draw[thick] (.5,-1.5) .. controls +(1.7,0) .. +(3.4,1.5);
\end{tikzpicture}
\end{align*}

We also give the tableau $T'$ corresponding to $m'$ via the map $f$ in Remark~\ref{fgeneral} (notice $T'$ is no longer standard).
\begin{align*}
T'= \begin{ytableau}
*(red) & *(red)  &*(red) & *(red) & *(red) & *(red) &*(red) & *(red)& *(red)& *(red)& *(red)& *(red) & b\\
*(red) & *(red)  &*(red) & *(red) & a & \cdot  &\cdot & \underline{b} & b''& \cdot  &\cdot & 2n &b'\\
\end{ytableau}
\end{align*}

Now we use the induction hypothesis and argue that there is a method of resolving crossings in $m_2$, in order to obtain $m$. First, remove the red cup $(b,b+1)$ in $w$ and $m'$. Since this cup does not interfere with  other parts of the diagram, we only need to focus on the subdiagrams in $w$ and $m'$ after removal. Call these $w|_{n-1}$ and $m'|_{n-1}$, respectively. Let their corresponding tableaux be $S|_{n-1}$ and $T'|_{n-1}$, respectively. Here $w|_{n-1}$ and $S|_{n-1}$ are identified via $\phi$, whereas $m'|_{n-1}$ and $T'|_{n-1}$ are identified via $f$. Use $N=2n-2$, we have 
\begin{align*}
&w|_{n-1}= \begin{tikzpicture}[baseline={(0,-.3)}, scale = 0.8]
\draw[dotted] (-.25,0) -- (5,0) -- (5,-1.7) -- (-.25,-1.7) -- cycle;
\begin{footnotesize}
\node at (.9,.2) {$a$};
\node at (1.1,-.2) {\dots};
\node at (1.9,.2) {$b$};
\node at (2.3,.2) {$b'$};
\node at (3,-.2) {\dots};
\node at (4.1,.2) {$2n-2$};
\end{footnotesize}
\draw[thick] (.5,-.9) .. controls +(.6,0) .. +(1.3,.9);
\draw[thick] (.5,-1.2) .. controls +(.8,0) .. +(1.7,1.2);
\draw[thick] (.5,-1.5) .. controls +(1.7,0) .. +(3.4,1.5);
\end{tikzpicture}\:, \hspace{.3 in}
m'|_{n-1}= \begin{tikzpicture}[baseline={(0,-.3)}, scale = 0.8]
\draw[dotted] (-.25,0) -- (5,0) -- (5,-1.7) -- (-.25,-1.7) -- cycle;
\begin{footnotesize}
\node at (.9,.2) {$a$};
\node at (1.1,-.2) {\dots};
\node at (1.9,.2) {$b$};
\node at (2.3,.2) {$b'$};
\node at (3,-.2) {\dots};
\node at (4.1,.2) {$2n$};
\end{footnotesize}
\draw[thick] (.5,-.6) .. controls +(.15,0) .. +(.3,.6);
\draw[thick] (.5,-.9) .. controls +(.7,0) .. +(1.4,.9);
\draw[thick] (.5,-1.2) .. controls +(.9,0) .. +(1.8,1.2);
\draw[thick] (.5,-1.5) .. controls +(1.7,0) .. +(3.4,1.5);
\end{tikzpicture}
\\
&S|_{n-1}=\begin{ytableau}
*(blue) & *(blue)  &*(blue) & *(blue) & *(blue) & *(blue) & *(blue) & *(blue) \\
*(blue) & *(blue) &*(blue) & *(blue) & b & \cdot & \cdot & N\\
\end{ytableau}\:, \hspace{.5 in}
T'|_{n-1}  \begin{ytableau}
*(red) & *(red)  &*(red) & *(red) & *(red)  & *(red)& *(red)& *(red)\\
*(red) & *(red)  &*(red) & *(red) & a & \cdot  &\cdot  & N \\
\end{ytableau}
\end{align*}
One key fact note is that the blue (and red, respectively) region in the tableau remains the same. Therefore $S|_{n-1}$ and $T'|_{n-1}$ satisfy the requirement in the induction hypothesis, and there is a path from $m'|_{n-1}$ to $w|_{n-1}$ by resolving the crossings using suitable move. The path from $w$ to $m$ is established by using the sequence of moves from $m$ to $m'$, then from $m'$ to $w$.
\end{proof}

Conjecture~\ref{mainconjecture} (\cite[Conjecture~5.8]{RT19}) is a direct consequence of the above result.
\begin{theorem}\label{mainth}
The entries in the transitioning matrix satisfy $a_{ST}>0$ if and only if $S\leq T$.
\end{theorem}
\begin{proof}
The ``only if'' direction is indicated by Theorem~\ref{unitriangularity}. To prove the ``if'' direction, notice $S\leq T$ implies that $a_i>b_i$ for $1\leq i\leq n$ by Corollary~\ref{entries2}, if $a_i$, $b_i$ are entries in the first rows of $S$ and $T$, respectively. Proposition~\ref{existspath} then implies that there is a crossing-resolving graph $\Gamma$ associated to $f(v_T)$, in which $\phi(S)$ occurs as a sink. Lemma~\ref{formula} then implies that the path count $a_{ST}$ is positive.
\end{proof}
\begin{remark}
Note we are not choosing a $\Gamma$ uniformly to deduce that all $S$ occurs as a sink in $\Gamma$ for $S\leq T$. Rather, for each $S$, the path in Proposition~\ref{existspath} is constructed according to $S$ (and $T$), resulting in a $\Gamma$ containing this path of resolving crossings.
\end{remark}

\subsection{An example}

We give one example to further illustrate the crossing-resolving algorithm in Proposition~\ref{existspath}.
\begin{ex}\label{hugeex}
Let $n=5$,
\begin{align*}
\ytableausetup{smalltableaux,notabloids}
S= \begin{ytableau}
1 & 3 &  4 & 6 & 9\\
2  &5& 7& 8& 10
\end{ytableau}, \hspace{.5 in} 
T=s_7s_8s_5s_2S=
 \begin{ytableau}
1 & 2 &  4 & 5 & 7\\
3  &6& 8& 9& 10
\end{ytableau}. 
\end{align*}
The corresponding web and matching are
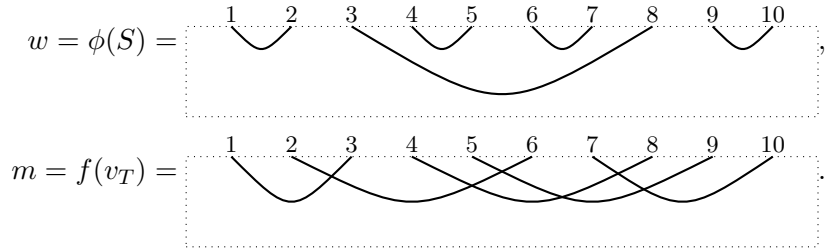
\begin{align*}
w=\phi(S)=\begin{tikzpicture}[baseline={(0,-.3)}, scale = 0.8]
\draw[dotted] (-.25,0) -- (10.25,0) -- (10.25,-1.5) -- (-.25,-1.5) -- cycle;
\begin{footnotesize}
\node at (.5,.2) {$1$};
\node at (1.5,.2) {$2$};
\node at (2.5,.2) {$3$};
\node at (3.5,.2) {$4$};
\node at (4.5,.2) {$5$};
\node at (5.5,.2) {$6$};
\node at (6.5,.2) {$7$};
\node at (7.5,.2) {$8$};
\node at (8.5,.2) {$9$};
\node at (9.5,.2) {$10$};
\end{footnotesize}
\draw[thick] (.5,0) .. controls +(.5,-.5)  .. +(1,0);
\draw[thick] (3.5,0) .. controls +(.5,-.5)  .. +(1,0);
\draw[thick] (5.5,0) .. controls +(.5,-.5)  .. +(1,0);
\draw[thick] (8.5,0) .. controls +(.5,-.5)  .. +(1,0);
\draw[thick] (2.5,0) .. controls +(2.5,-1.5)  .. +(5,0);
\end{tikzpicture}, \\
m=f(v_T)=
\begin{tikzpicture}[baseline={(0,-.3)}, scale = 0.8]
\draw[dotted] (-.25,0) -- (10.25,0) -- (10.25,-1.5) -- (-.25,-1.5) -- cycle;
\begin{footnotesize}
\node at (.5,.2) {$1$};
\node at (1.5,.2) {$2$};
\node at (2.5,.2) {$3$};
\node at (3.5,.2) {$4$};
\node at (4.5,.2) {$5$};
\node at (5.5,.2) {$6$};
\node at (6.5,.2) {$7$};
\node at (7.5,.2) {$8$};
\node at (8.5,.2) {$9$};
\node at (9.5,.2) {$10$};
\end{footnotesize}
\draw[thick] (.5,0) .. controls +(1,-1)  .. +(2,0);
\draw[thick] (1.5,0) .. controls +(2,-1)  .. +(4,0);
\draw[thick] (3.5,0) .. controls +(2,-1)  .. +(4,0);
\draw[thick] (4.5,0) .. controls +(2,-1)  .. +(4,0);
\draw[thick] (6.5,0) .. controls +(1.5,-1)  .. +(3,0);
\end{tikzpicture}. 
\end{align*}
The path of crossing-resolving moves from $m$ to $w$ is given by the following. Here we use colors so the reader can keep track of the strands more easily:
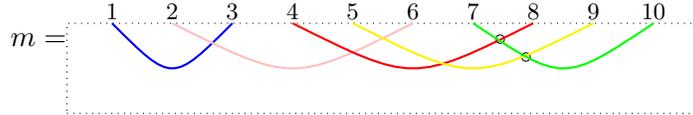
\begin{align*}
m= &\begin{tikzpicture}[baseline={(0,-.3)}, scale = 0.8]
\draw[dotted] (-.25,0) -- (10.25,0) -- (10.25,-1.5) -- (-.25,-1.5) -- cycle;
\begin{footnotesize}
\node at (.5,.2) {$1$};
\node at (1.5,.2) {$2$};
\node at (2.5,.2) {$3$};
\node at (3.5,.2) {$4$};
\node at (4.5,.2) {$5$};
\node at (5.5,.2) {$6$};
\node at (6.5,.2) {$7$};
\node at (7.5,.2) {$8$};
\node at (8.5,.2) {$9$};
\node at (9.5,.2) {$10$};
\node at(6.95,-.28) {$\circ$};
\node at(7.38,-.57) {$\circ$};
\end{footnotesize}
\draw[thick,color=blue] (.5,0) .. controls +(1,-1)  .. +(2,0);
\draw[thick,color=pink] (1.5,0) .. controls +(2,-1)  .. +(4,0);
\draw[thick,color=red] (3.5,0) .. controls +(2,-1)  .. +(4,0);
\draw[thick,color=yellow] (4.5,0) .. controls +(2,-1)  .. +(4,0);
\draw[thick,color=green] (6.5,0) .. controls +(1.5,-1)  .. +(3,0);
\end{tikzpicture}. 
\end{align*}
To move the green cup to (9,10), we swap 7 and 8 (Type VV), then swap 8 and 9 (Type VV):
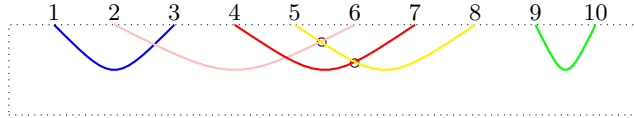
\begin{align*}
&\begin{tikzpicture}[baseline={(0,-.3)}, scale = 0.8]
\draw[dotted] (-.25,0) -- (10.25,0) -- (10.25,-1.5) -- (-.25,-1.5) -- cycle;
\begin{footnotesize}
\node at (.5,.2) {$1$};
\node at (1.5,.2) {$2$};
\node at (2.5,.2) {$3$};
\node at (3.5,.2) {$4$};
\node at (4.5,.2) {$5$};
\node at (5.5,.2) {$6$};
\node at (6.5,.2) {$7$};
\node at (7.5,.2) {$8$};
\node at (8.5,.2) {$9$};
\node at (9.5,.2) {$10$};
\node at(4.95,-.3) {$\circ$};
\node at(5.5,-.64) {$\circ$};
\end{footnotesize}
\draw[thick,color=blue] (.5,0) .. controls +(1,-1)  .. +(2,0);
\draw[thick,color=pink] (1.5,0) .. controls +(2,-1)  .. +(4,0);
\draw[thick,color=red] (3.5,0) .. controls +(1.5,-1)  .. +(3,0);
\draw[thick,color=yellow] (4.5,0) .. controls +(1.5,-1)  .. +(3,0);
\draw[thick,color=green] (8.5,0) .. controls +(.5,-1)  .. +(1,0);
\end{tikzpicture}. 
\end{align*}
Now we move the yellow cup to (6,7). To do this we swap 5 and 6 (Type VV), then swap 7 and 8 (Type $\mathbb{V}$):
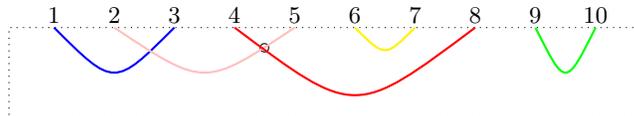
\begin{align*}
&\begin{tikzpicture}[baseline={(0,-.3)}, scale = 0.8]
\draw[dotted] (-.25,0) -- (10.25,0) -- (10.25,-1.5) -- (-.25,-1.5) -- cycle;
\begin{footnotesize}
\node at (.5,.2) {$1$};
\node at (1.5,.2) {$2$};
\node at (2.5,.2) {$3$};
\node at (3.5,.2) {$4$};
\node at (4.5,.2) {$5$};
\node at (5.5,.2) {$6$};
\node at (6.5,.2) {$7$};
\node at (7.5,.2) {$8$};
\node at (8.5,.2) {$9$};
\node at (9.5,.2) {$10$};
\node at(4,-.35) {$\circ$};
\end{footnotesize}
\draw[thick,color=blue] (.5,0) .. controls +(1,-1)  .. +(2,0);
\draw[thick,color=pink] (1.5,0) .. controls +(1.5,-1)  .. +(3,0);
\draw[thick,color=red] (3.5,0) .. controls +(2,-1.5)  .. +(4,0);
\draw[thick,color=yellow] (5.5,0) .. controls +(.5,-.5)  .. +(1,0);
\draw[thick,color=green] (8.5,0) .. controls +(.5,-1)  .. +(1,0);
\end{tikzpicture}. 
\end{align*}
Now we move the red cup to (4,5). This requires one Type $\mathbb{V}$ move of swapping 5 and 8:
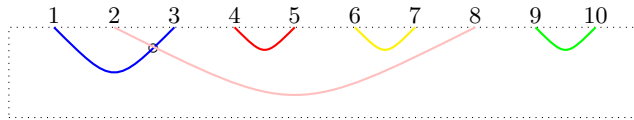
\begin{align*}
\begin{tikzpicture}[baseline={(0,-.3)}, scale = 0.8]
\draw[dotted] (-.25,0) -- (10.25,0) -- (10.25,-1.5) -- (-.25,-1.5) -- cycle;
\begin{footnotesize}
\node at (.5,.2) {$1$};
\node at (1.5,.2) {$2$};
\node at (2.5,.2) {$3$};
\node at (3.5,.2) {$4$};
\node at (4.5,.2) {$5$};
\node at (5.5,.2) {$6$};
\node at (6.5,.2) {$7$};
\node at (7.5,.2) {$8$};
\node at (8.5,.2) {$9$};
\node at (9.5,.2) {$10$};
\node at(2.15,-.35) {$\circ$};
\end{footnotesize}
\draw[thick,color=blue] (.5,0) .. controls +(1,-1)  .. +(2,0);
\draw[thick,color=red] (3.5,0) .. controls +(.5,-.5)  .. +(1,0);
\draw[thick,color=yellow] (5.5,0) .. controls +(.5,-.5)  .. +(1,0);
\draw[thick,color=green] (8.5,0) .. controls +(.5,-.5)  .. +(1,0);
\draw[thick,color=pink] (1.5,0) .. controls +(3,-1.5)  .. +(6,0);
\end{tikzpicture}. 
\end{align*}
Lastly we move the pink cup to (3,8), via a Type VV move of swapping 2 and 3:
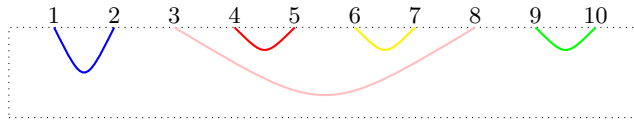
\begin{align*}
\begin{tikzpicture}[baseline={(0,-.3)}, scale = 0.8]
\draw[dotted] (-.25,0) -- (10.25,0) -- (10.25,-1.5) -- (-.25,-1.5) -- cycle;
\begin{footnotesize}
\node at (.5,.2) {$1$};
\node at (1.5,.2) {$2$};
\node at (2.5,.2) {$3$};
\node at (3.5,.2) {$4$};
\node at (4.5,.2) {$5$};
\node at (5.5,.2) {$6$};
\node at (6.5,.2) {$7$};
\node at (7.5,.2) {$8$};
\node at (8.5,.2) {$9$};
\node at (9.5,.2) {$10$};
\end{footnotesize}
\draw[thick,color=blue] (.5,0) .. controls +(.5,-1)  .. +(1,0);
\draw[thick,color=red] (3.5,0) .. controls +(.5,-.5)  .. +(1,0);
\draw[thick,color=yellow] (5.5,0) .. controls +(.5,-.5)  .. +(1,0);
\draw[thick,color=green] (8.5,0) .. controls +(.5,-.5)  .. +(1,0);
\draw[thick,color=pink] (2.5,0) .. controls +(2.5,-1.5)  .. +(5,0);
\end{tikzpicture}. 
\end{align*}
This way we have obtained the web $w$.

\section{The inverse matrix}
This section is dedicated to investigating the inverse matrix for $(a_{ST})$ in Theorem~\ref{unitriangularity}, which may serve as a first step in understanding the properties of this matrix. In certain settings, the entries of this matrix are known to be related to coefficients in cohomology groups \cite{SW19}, in the study of Springer fibers of the Steinberg variety.

Let $\psi: W^{\lambda} \to S^{\lambda}$ be the inverse of $\rho$ introduced in Section~\ref{S2.3}. Recall the Young tableau $T_0$ also introduced in Section~\ref{S2.3}, and let $w_0=\phi(T_0)$. In particular, $w_0$ is the web with $n$ cups side by side.
\end{ex}

\begin{lemma}\label{KLontableau}
If $w=rw_0$ where $r=s_{i_1}s_{i_2}\cdots s_{i_t}$ is a reduced word obtained via a path from $w_0$ to $w$, then
\begin{align*}
\psi(w)=(s_{i_1}-1)(s_{i_2}-1)\cdots (s_{i_t}-1)v_{T_0}. 
\end{align*}
\end{lemma}
\begin{proof}
This follows from the fact that whenever $R \overset{s_i} \longrightarrow U$ is a directed edge in the tableau graph, then $U=(t_i-1)R$. The claim then follows from the action of $S_{2n}$ on the cup diagrams in (\ref{webaction}).
\end{proof}

The polytabloids satisfy certain relations, known as the Garnir relations, as given in \cite{Fu97,Gre80,Sag01}. The following is a consequence of the theorem in the special case when the relevant partition is $(n,n)$.
\begin{theorem} \label{Garnir}
1) If $i$ and $j$ are in the same column of $T$, then
\begin{align}
(i \hspace{.1 in} j). v_T=-v_T.  \label{R1}
\end{align}

2) If $T$ has shape $(n,n)$ and $S$ is obtained from $T$ by interchanging two given columns, then 
\begin{align}
v_T=v_S.   \label{R2}
\end{align}

3)  If $T$ has shape $(n,n)$, then
\begin{align}
\ytableausetup{smalltableaux,notabloids}
v_{\begin{ytableau}
*(white) & &a& c & &\\
 & & b&  x &  &
\end{ytableau}}
&=
v_{\begin{ytableau}
*(white) && c & a && \\
&& b & x &&
\end{ytableau}} +
v_{\begin{ytableau}
*(white)&& a & b &&\\
&& c & x &&
\end{ytableau}}. \label{R3}
\end{align}
Here, blank boxes represent an arbitrary number of boxes, whose entries remain the same throughout the equation.
\end{theorem}
%=============================================

%Recall that the partial order on $\mathcal{T}(n,n)$ is ranked, where the rank function is given by $\ell(w)$ for $T=w.T_0$. In \cite{RT19}, another computation of this rank is given equivalently as the nesting number for the associated web $w=\phi(T)$. In particular, an arc $(a,b)$ is called \emph{nested inside} another arc $(a',b')$, if $a'<a<b<b'$. The \emph{nesting number} of a web $w$ is the number of pairs $(A_1,A_2)$, where $A_1$, $A_2$ are both arcs in $w$, and $A_1$ is nested inside $A_2$. An example of the computation 

\begin{proposition}\label{explicitpreimage}
We have $\psi(w)=v_T$, where $T$ is a tableau such that both entries in each column of $T$ are connected by an arc in $w$, and the smaller entry is on top. Moreover, $v_T$ is well-defined by this condition. 
\end{proposition}
\begin{proof}
By \eqref{R2}, one can always rearrange the columns such that $v_T$ remains the same, hence $v_T$ is well-defined. We induct on the partial order on the set of webs. Suppose the statement  holds for all webs of order less than or equal to $w'$. Given $w$ which is immediately grater than $w'$, i.e. there is an edge $w'\to w$ via the simple transposition $s_i$. We have $w=(s_i-1)w'$ and 
\begin{align*}
\psi(w)=(s_i-1)\psi(w')=(s_i-1)v_{T'},
\end{align*}
where $T'$ is a tableau obtained from $w'$ via the condition stated in the lemma. 

Since $i$ is a right endpoint in $w'$ and $i+1$ a left endpoint in $w'$, let $x$ be the left endpoint of the arc incident to $i$, and $y$ be the right endpoint of the arc incident to $i+1$. That is to say,
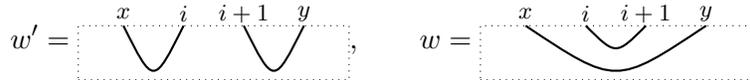
\begin{align*}
w'=\begin{tikzpicture}[baseline={(0,-.3)}, scale = 0.8]
\draw[dotted] (-.25,0) -- (4.25,0) -- (4.25,-.9) -- (-.25,-.9) -- cycle;
\begin{footnotesize}
\node at (.5,.2) {$x$};
\node at (1.5,.2) {$i$};
\node at (2.5,.2) {$i+1$};
\node at (3.5,.2) {$y$};
\end{footnotesize}
\draw[thick] (.5,0) .. controls +(.5,-1)  .. +(1,0);
\draw[thick] (2.5,0) .. controls +(.5,-1)  .. +(1,0);
\end{tikzpicture},  
\hspace{.3 in}
w=\begin{tikzpicture}[baseline={(0,-.3)}, scale = 0.8]
\draw[dotted] (-.25,0) -- (4.25,0) -- (4.25,-.9) -- (-.25,-.9) -- cycle;
\begin{footnotesize}
\node at (.5,.2) {$x$};
\node at (1.5,.2) {$i$};
\node at (2.5,.2) {$i+1$};
\node at (3.5,.2) {$y$};
\end{footnotesize}
\draw[thick] (.5,0) .. controls +(1.5,-1)  .. +(3,0);
\draw[thick] (1.5,0) .. controls +(.5,-.5)  .. +(1,0);
\end{tikzpicture}. 
\end{align*}
Let $j=i+1$, then
\begin{align*}
T'= \begin{ytableau}
*(white)&&x&&&& j &&\\
&& i &&&& y &&
\end{ytableau}. 
\end{align*}
Here blank boxes represent an arbitrary number of boxes. Hence
\begin{align*}
\psi(w)=&(s_i-1)v_{T'}\\
=&v_{ \begin{ytableau}
*(white)&&x&&&& i &&\\
&& j &&&& y &&
\end{ytableau}}-
v_{ \begin{ytableau}
*(white)&&x&&&& j &&\\
&& i &&&& y &&
\end{ytableau}}\\
=&v_{ \begin{ytableau}
*(white)&&x&&&& i &&\\
&& j &&&& y &&
\end{ytableau}}+
v_{ \begin{ytableau}
*(white)&&i&&&& j &&\\
&& x &&&& y &&
\end{ytableau}}=
v_{ \begin{ytableau}
*(white)&&i&&&& x &&\\
&& j &&&& y &&
\end{ytableau}}
\end{align*}
is in the form desired. The other columns in $T$ and $T'$ coincide because the other arcs in $w'$ and $w$ coincide. 
\end{proof}

\begin{remark}
Because the standard polytabloids span $S^{\lambda}$, using the Garnir relations in Theorem~\ref{Garnir}, one is guaranteed to expand $v_T$ as a linear combination of standard polytabloids. Hence Proposition~\ref{explicitpreimage} gives an indirect algorithm of computing the coefficients in $\psi(w)$ for any cup diagram $w$.
\end{remark}

\begin{ex}
Let $w=\phi(S)$ be the same as in Example~\ref{hugeex}. Then $\psi(w)=v_R$, where
\begin{align*}
R=\begin{ytableau}
1 & 3 & 4 & 6 & 9 \\
2 & 8 & 5 & 7 & 10
\end{ytableau}. 
\end{align*}
Note that $R$ is no longer standard.
\end{ex}

\bibliography{litlist} \label{references}
\bibliographystyle{amsalpha}

\end{document}